\documentclass[letterpaper, 10 pt, conference]{ieeeconf}  

\IEEEoverridecommandlockouts                              
\usepackage{amsmath,amssymb}

\usepackage{amsthm}
\usepackage[mathscr]{eucal}
\usepackage{bbm}
\usepackage{color}
\usepackage{todonotes}
\usepackage{bm}
\usepackage{soul}
\usepackage[ruled,vlined]{algorithm2e}
\usepackage{pgfgantt}
\usepackage{graphicx}
\usepackage{xcolor}
\usepackage{datetime}
\usepackage{standalone}
\usepackage{dsfont}
\newtheorem{Theorem}{\bf Theorem}
\newtheorem{Definition}[Theorem]{\bf Definition} 
\newtheorem{Lemma}[Theorem]{\bf Lemma}

\newtheorem{Assumption}[Theorem]{\bf Assumption}

\usepackage{algorithmic}

\newdateformat{monthdayyeardate}{\monthname[\THEMONTH]~\THEDAY, \THEYEAR}
\newdateformat{monthdayyeardatee}{\THEDAY.\THEMONTH.\THEYEAR}
\newcommand{\minus}{\scalebox{0.55}[1.0]{$-$}}

\newcommand\scalemath[2]{\scalebox{#1}{\mbox{\ensuremath{\displaystyle #2}}}}

\title{\LARGE \bf
	Data-driven distributionally robust MPC \\for systems with uncertain dynamics
}
\author{Francesco Micheli, Tyler Summers, John Lygeros
	\thanks{Research supported by the European Research Council under the H2020 Advanced Grant no. 787845 (OCAL).}%
	\thanks{F. Micheli and J. Lygeros are with the Automatic Control Laboratory in the Department of Information Technology and Electrical Engineering, ETH Z\"{u}rich, Switzerland. Emails: \texttt{\{frmicheli, jlygeros\}@ethz.ch}.}%
	\thanks{T. Summers is with the Department of Mechanical Engineering at University of Texas at Dallas, USA. Email: \texttt{tyler.summers@utdallas.edu}. The work of T. Summers is supported by the United States Air Force Office of Scientific Research under award number FA2386-19-1-4073 and the National Science Foundation under award number ECCS-2047040.}%
}

\begin{document}
	
	\maketitle
	\thispagestyle{empty}
	\pagestyle{empty}

	\begin{abstract}
		We present a novel data-driven distributionally robust Model Predictive Control formulation for unknown discrete-time linear time-invariant systems affected by unknown and possibly unbounded additive uncertainties. We use off-line collected data and an approximate model of the dynamics to formulate a finite-horizon optimization problem. To account for both the uncertainty related to the dynamics and the disturbance acting on the system, we resort to a distributionally robust formulation that optimizes the cost expectation while satisfying Conditional Value-at-Risk constraints with respect to the worst-case probability distributions of the uncertainties within an ambiguity set defined using the Wasserstein metric.
		Using results from the distributionally robust optimization literature we derive a tractable finite-dimensional convex optimization problem with finite-sample guarantees for the class of convex piecewise affine cost and constraint functions.
		The performance of the proposed algorithm is demonstrated in closed-loop simulation on a simple numerical example.
	\end{abstract}
	
	\section{Introduction}
	Model Predictive Control (MPC) relies on a model of the system dynamics to repeatedly solve, at each time-step, a Finite-Horizon Optimal Control (FHOC) problem subject to input and state constraints. Control performance is related to the quality of the open-loop predictions, making MPC susceptible to uncertainties in the prediction model.
	Uncertainty is present in many practically relevant control applications, as both the dynamics and the disturbance distribution are only approximately known or must be estimated from data. Hence, in the last two decades, robust and stochastic MPC have been developed to cope explicitly with uncertainty whenever the robustness that might be implicitly provided by feedback from deterministic MPC is insufficient.
	
	Robust MPC (RMPC) \cite{mayne2000constrained} relies on the assumption of bounded uncertainties to solve a worst-case optimization problem.  Tube-based formulations \cite{langson2004robust,munoz2013recursively} have been developed to account for both the uncertainty on the model and on the disturbance realizations. Since RMPC accounts for all possible realization of the uncertainties, thus neglecting any available distributional information, it often results in a conservative controller.
	
	When distributional information on the model uncertainty and disturbance is available, Stochastic MPC (SMPC) can reduce the conservatism of RMPC by enforcing constraints to hold in probability \cite{cannon2009probabilistic,cannon2010stochastic}. Unfortunately, analytical solutions to SMPC are only available under specific assumptions on the uncertainties distributions. Alternatively, randomized methods such as the sample average approximation~\cite{kleywegt2002sample} and the scenario approach~\cite{calafiore2006scenario} can be employed to reformulate the stochastic problem into a large, but finite dimensional, deterministic program. Since randomized methods leverage sampling, they can handle generic distributions and can be applied when these are only accessible through sampling.  
	Applications of the scenario approach for a SMPC with uncertain dynamics and unbounded stochastic disturbance are analyzed in \cite{micheli2022scenario} and \cite{calafiore2012robust}, where the scenarios are obtained by sampling the dynamics and the disturbance distributions.
	
	Crucially, RMPC and SMPC rely on specific assumptions on the disturbances or, in the case of the randomized methods, availability of large amount of data is required to provide tight probabilistic guarantees. In many situations, we might only have access to an approximate dynamics and limited amount of data regarding disturbance process, resulting in uncertainty on the dynamics and on disturbance distribution. We therefore resort to a Distributionally Robust MPC (DRMPC) formulation that solves a stochastic optimization problem with respect to the worst-case probability distribution within an ambiguity set. This allows to avoid the excessive conservatism of RMPC while protecting against the distributional mismatch that can negatively impact SMPC.
	
	Lately, a number of works developed DRMPC formulation to robustify against additive disturbances or uncertain constraints considering moment-based \cite{van2015distributionally,coppens2021data,li2021distributionally} and Wasserstein-based \cite{mark2020stochastic,zhong2021data,zolanvari2021data} uncertainty sets.  
	None of these approaches directly addresses the issue of uncertainty in the system dynamics parameters; they instead focus on a limited uncertainty representation from additive noise only. This is a fundamental aspect to consider whenever the system is too complex to be precisely modeled and an approximate model is obtained from limited data or when the controller needs to be robust against plant changes resulting from production variability or system aging.
	
	In this work we propose a data-driven DRMPC formulation that robustifies against both the uncertainty on the approximate dynamics and the uncertainty on the additive disturbance acting on the system. 
	Unlike earlier works on DRMPC, we directly robustify against the distributional mismatch in the predicted state trajectories. We do so by defining the ambiguity set as a ball centered on the empirical distribution given by the predictions obtained with the uncertain dynamics and radius defined with respect to the Wasserstein metric. This allows us to capture both the uncertainty in the approximate dynamics and the uncertainty on the disturbance distributions.
	Leveraging earlier results from the distributionally robust literature we derive finite sample guarantees and present tractable finite-dimensional convex formulations for the DRMPC problem with worst-case expectation cost and Conditional Value-at-Risk constraints.
	The derived formulation generalizes previous works on DRMPC that only address the effect of the additive disturbance and consider known dynamics. Whenever the dynamics can be considered exactly known, the formulation reduces to a DRMPC with only additive uncertainty, highlighting the consistency of the approach.
	
	\textit{Outline:} The remainder of this paper is organized as follows. In Section~\ref{Sec:Problem_Formulation} we introduce the problem setting and formally state the distributionally robust control problem. In Section~\ref{Sec:MainResults} we derive finite sample guarantees for the DRMPC problem with worst-case expectation cost and Conditional Value-at-Risk constraints. Section~\ref{Sec:Numerics} demonstrates the effectiveness of the proposed algorithm on a simple closed-loop example and Section~\ref{Sec:Conclusion} concludes the paper.
	
	\textit{Notation:} We denote by $\delta_x$ the Dirac distribution at $x$. We denote by $\| \cdot\|:=\| \cdot \|_2$ the standard Euclidean distance, for matrix norms the same symbol is used to denote the norm induced by the $2$-norm.  We define by $(\cdot)_+ := \max\{\cdot,0\}$.
	
	\section{Problem formulation}\label{Sec:Problem_Formulation}
	Consider the discrete-time linear time-invariant (LTI) system subject to additive disturbances
	\begin{equation*}\label{eq:Dynamics}\scalemath{0.84}{
		x_{k+1}=\bar{A} x_{k} + \bar{B} u_{k} +w_{k}\ ,}
	\end{equation*}
	with state $x_k \in \mathbb{R}^{n}$, control input $u_k \in \mathbb{R}^{m}$ and disturbance  $w_k \in \mathbb{R}^{n}$ distributed according to a unknown probability distribution $\mathbb{P}_{w}$ over the unknown and possibly unbounded support set $\mathcal{W} \subseteq \mathbb{R}^{n}$.
	We assume that the state is measured, but that we do not have access to the true system dynamics.
	
	We introduce the compact $T$-step formulation of the dynamics
	\begin{equation}\label{eq:dynamics_matrix_form}\scalemath{0.84}{
		\boldsymbol{y}=\bar{L} \boldsymbol{z} + {\boldsymbol{\xi}}}
	\end{equation}
	with
	\begin{equation*}\scalemath{0.84}{
		\boldsymbol{y}=\!\begin{bmatrix}
			x_{1} \\
			x_{2} \\
			\vdots \\
			x_{T}
		\end{bmatrix}, \ \boldsymbol{u}=\!\begin{bmatrix}
			u_{0} \\
			u_{1} \\
			\vdots \\
			u_{T\!-\!1}
		\end{bmatrix}, \ \boldsymbol{w}=\!\begin{bmatrix}
			w_{0} \\
			w_{1} \\
			\vdots \\
			w_{T\!-\!1}
		\end{bmatrix}, \ \boldsymbol{z} =\! \begin{bmatrix}	\boldsymbol{x}_0 \\ \boldsymbol{u} \end{bmatrix},\ {\boldsymbol{\xi}} =\! \bar{H} \boldsymbol{w},}
	\end{equation*}
	\begin{equation*}\scalemath{0.8}{
			\bar{{L}}\!=\!\!
			\begin{bmatrix}
				\bar{A} & \bar{B} & 0_{n\! \times\! m} & \cdots & 0_{n\! \times\! m} \\
				\bar{A}^{2} & \bar{A} \bar{B} & \bar{B} & \ddots & \vdots \\
				\vdots & \vdots & \ddots & \ddots & 0_{n\! \times\! m} \\
				\bar{A}^{T} & \bar{A}^{T\!\minus\!1}\! \bar{B} & \cdots & \bar{A} \bar{B} & \bar{B}
			\end{bmatrix}\!,\, 
		\bar{{H}}\!=\!\!
			\begin{bmatrix}
				I_{n\! \times\! n} & 0_{n\! \times\! n} & \cdots & 0_{n\! \times\! n} \\
				\bar{A} & I_{n\! \times\! n} & \ddots & \vdots \\
				\vdots & \ddots & \ddots & 0_{n\! \times\! n} \\
				\bar{A}^{T\!\minus\!1} & \cdots & \bar{A} & I_{n\! \times\! n}
			\end{bmatrix}\!.}\vspace{5pt}
	\end{equation*}
	The future state sequence $\boldsymbol{y}$ depends on the initial condition $\boldsymbol{x}_0$, the control input sequence $\boldsymbol{u}$ and the multi-step disturbance ${\boldsymbol{\xi}} \in \Xi \subseteq \mathbb{R}^{nT}$ distributed according to $\mathbb{P}_{\boldsymbol{\xi}}$. We also define $\mathbb{P}_{\boldsymbol{y}|\boldsymbol{z}}$ as the probability distribution of $\boldsymbol{y}$ resulting from the dynamics~\eqref{eq:dynamics_matrix_form} for given initial condition and control input sequence.
	
	We are interested in designing a receding horizon predictive controller that minimizes a given cost function over a finite horizon $T$.
	At each time $\tau$ we would like to solve the finite horizon optimal control (FHOC) problem 
	\begin{equation}\label{eq:MPC_problem_2}\scalemath{0.84}{
		\begin{aligned}
			\bar{J}_{\tau} := \min_{\boldsymbol{z}} \quad & \mathbb{E}^{\boldsymbol{y}\sim\mathbb{P}_{\boldsymbol{y}|\boldsymbol{z}}}\left[ h\left(\boldsymbol{y}, \boldsymbol{z} \right) \right]\\
			\text{s.t.}\quad& \boldsymbol{z} \in  \mathcal{Z}_{\tau}\ \\
			\quad & \text{CVaR}_{1-\beta}^{\boldsymbol{y}\sim\mathbb{P}_{\boldsymbol{y}|\boldsymbol{z}}}\left(g(\boldsymbol{y}, \boldsymbol{z})\right) \leq 0  \ .
		\end{aligned}}
	\end{equation}
	The set $\mathcal{Z}_{\tau}$ includes a constraint that sets $\boldsymbol{x}_0$ equal to the measured value of the state at the time $\tau$ when~\eqref{eq:MPC_problem_2} is solved, and possibly other constraints that we want to impose on the input trajectory. Since the future states $\boldsymbol{y}$ are uncertain, we enforce the constraint $g\left( \boldsymbol{y},\boldsymbol{z}\right)\leq 0$ with $g\left( \boldsymbol{y},\boldsymbol{z}\right): \mathbb{R}^{nT}\times \mathbb{R}^{nT} \to \mathbb{R}$ as a Conditional Value-at-Risk (CVaR) constraint. 
	\begin{Definition}[Conditional Value-at-Risk]
		For a random variable $\omega\in\Omega\subseteq\mathbb{R}^r$ with distribution $\mathbb{P}_{\omega}$ and a function $\phi:\mathbb{R}^r\rightarrow\mathbb{R}$, the CVaR of level $\beta$ is defined as
		\begin{equation*}\scalemath{0.84}{
			\text{CVaR}_{1-\beta}^{\omega\sim\mathbb{P}_{\omega}}\left(\phi(\omega)\right):= \inf_{t \in \mathbb{R}} \left[{\beta}^{-1}\, \mathbb{E}^{\omega\sim\mathbb{P}_{\omega}}\left[\left(\phi(\omega) +t\right)_{+}\right]-t\right]\ .}
		\end{equation*}
	\end{Definition}
	\subsection{The sample average approximation}\label{Sec:subsec_SAA}
	Since $\bar{L}$ and $\mathbb{P}_{\boldsymbol{\xi}}$ are unknown, \eqref{eq:MPC_problem_2} cannot be solved directly. We assume the availability of a dataset $\mathcal{D}^{N,T}$ comprising $N$ $T$-step input-state trajectories $\{\boldsymbol{z}^i,\boldsymbol{y}^i\}$, $i=1,\dots,N$, collected by applying an input trajectory $\boldsymbol{u}^i$ of length $T$ starting from an initial condition $\boldsymbol{x}_0^i$. We also assume to have access to an approximate $T$-step linear predictor $\hat{L}$, obtained from a separate identification procedure, and make the following assumption
	\vspace{-3pt}
	\begin{Assumption}{}\label{Ass:Spectral_norm_bound}
		For a given confidence level $\alpha \in (0,1)$, let $\gamma(\alpha):(0,1)\rightarrow \mathbb{R}$ be a function such that
		\begin{equation*}\scalemath{0.84}{
			\mathbb{P}^{\mathcal{ID}}\left\{\left\|\bar{L}-\hat{L}\right\| \leq \gamma(\alpha)\right\} \geq 1-\alpha\ .}
		\end{equation*}
	With $\mathbb{P}^{\mathcal{ID}}$ the probability related to the identification process.
	\end{Assumption}
	With this assumption we essentially bound the mismatch between the approximate and true dynamics. 
	References~\cite{dean2020sample,sarkar2019near,simchowitz2018learning} provide conditions for Assumption~\ref{Ass:Spectral_norm_bound} to hold, when the prediction model is obtained by linear regression from historical data. Nonetheless, the resulting theoretical bounds can be quite conservative in practice~\cite{dean2020sample}.
	
	We can use the data in $\mathcal{D}^{N,T}$ and $\hat{L}$ to compute $N$ approximate $T$-step residuals
	\begin{equation}\label{eq:Multistep_Residuals}\scalemath{0.84}{
		\hat{\boldsymbol{\xi}}^i := \boldsymbol{y}^i - \hat{L} \boldsymbol{z}^i \ ,\ \ i=1,\dots,N\ .}
	\end{equation}
	Given some initial condition $\boldsymbol{x}_0$ and control input trajectory $\boldsymbol{u}$, we can use the multi-step predictor $\hat{L}$ and the multi-step residuals obtained in~\eqref{eq:Multistep_Residuals} to compute approximate multi-step predictions as
	\begin{equation}\label{eq:Multistep_Prediction}\scalemath{0.84}{
		\hat{\boldsymbol{y}}^i(\boldsymbol{z}) := \hat{L} \boldsymbol{z} + \hat{\boldsymbol{\xi}}^i\ , \ \ i=1,\dots,N\ .}
	\end{equation}
	
	In the spirit of the Sample Average Approximation (SAA)~\cite{kleywegt2002sample}, we can approximate the solution of~\eqref{eq:MPC_problem_2} by
	\begin{equation}\label{eq:MPC_problem_SAA}\scalemath{0.84}{
		\begin{aligned}
			\hat{J}^{SAA} := \min_{\boldsymbol{z}} \quad & \mathbb{E}^{\boldsymbol{y}\sim\hat{\mathbb{P}}(\boldsymbol{z})}\left[ h\left(\boldsymbol{y}, \boldsymbol{z} \right) \right]\\
			\text{s.t.}\quad& \boldsymbol{z} \in  \mathcal{Z}_{\tau}\ \\
			\quad & \text{CVaR}_{1-\beta}^{\boldsymbol{y}\sim\hat{\mathbb{P}}(\boldsymbol{z})}\left(g(\boldsymbol{y}, \boldsymbol{z})\right) \leq 0\ ,
		\end{aligned}}
	\end{equation}
	with
	\begin{equation*}\scalemath{0.84}{
		\hat{\mathbb{P}}(\boldsymbol{z}):=
		\frac{1}{N} \sum_{i=1}^{N} \delta_{\hat{\boldsymbol{y}}^i(\boldsymbol{z})}}
	\end{equation*}
the empirical distribution of $\hat{\boldsymbol{y}}^i(\!\boldsymbol{z}\!),\, i\!=\!1,\dots,N$, as in~\eqref{eq:Multistep_Prediction}.
	
	\subsection{The distributionally robust formulation}\label{Sec:DR_FHOC}
	While simple to implement,~\eqref{eq:MPC_problem_SAA} can lead to poor out-of-sample performance and cannot provide tight probabilistic guarantees if we only have access to a limited amount of data or when there is a mismatch between the approximate and true dynamics. Thus, we resort to a distributionally robust formulation of~\eqref{eq:MPC_problem_2} to robustify against both the uncertainty related to the approximate dynamics and the disturbance acting on the system by optimizing the worst-case expectation within an ambiguity set defined using the Wasserstein metric. While our work generalizes to the so-called p-Wasserstein metric with an arbitrary norm, we restrict our attention to the 1-Wasserstein metric with the Euclidean norm.
	\vspace{-3pt}
	\begin{Definition}[Wasserstein metric~\cite{kantorovich1958space,mohajerin2018data}]
		Consider distributions $\mathbb{Q}_1,\mathbb{Q}_2 \in \mathcal{M}(\mathcal{Y})$, where $\mathcal{M}(\mathcal{Y})$ is the set of all probability distributions $\mathbb{Q}$ supported on $\mathcal{Y}\subseteq\mathbb{R}^{nT}$ such that $\mathbb{E}\left[\| \boldsymbol{y} \|\right] < \infty $.
		The Wasserstein metric $d_{\mathrm{W}} : \mathcal{M}(\mathcal{Y}) \times \mathcal{M}(\mathcal{Y}) \rightarrow \mathbb{R}_{\geq 0}$ between the distributions $\mathbb{Q}_1$ and $\mathbb{Q}_2$ is defined as
		\begin{equation}\label{eq:Wasserstein_distance}\scalemath{0.84}{
			d_{W}\left(\mathbb{Q}_{1}, \mathbb{Q}_{2}\right):=\inf \left\{\int_{\mathcal{Y}^{2}}\left\|\boldsymbol{y}_{1}-\boldsymbol{y}_{2}\right\| \Pi\left(\mathrm{d} \boldsymbol{y}_{1}, \mathrm{d} \boldsymbol{y}_{2}\right)\right\},}
		\end{equation}
		where $\Pi$ is the joint distribution of $\boldsymbol{y}_1$ and $\boldsymbol{y}_2$ with marginals $\mathbb{Q}_1$ and $\mathbb{Q}_2$.
	\end{Definition}
	
	The distributionally robust (DR) FHOC problem can now be written as
	\begin{equation}\label{eq:DR_MPC_1}\scalemath{0.84}{
		\begin{aligned}
			\hat{J}^{DR} := \min_{\boldsymbol{z}}  &\sup_{\mathbb{Q} \in \mathcal{B}^{\varepsilon}\left(\hat{\mathbb{P}}(\boldsymbol{z})\right)} \mathbb{E}^{\boldsymbol{y}\sim\mathbb{Q}}\left[ h\left(\boldsymbol{y}, \boldsymbol{z} \right) \right]\\
			\text{s.t.} \quad& \quad\boldsymbol{z} \in  \mathcal{Z}_{\tau}\ \\
			\quad & \sup_{\mathbb{Q} \in \mathcal{B}^{\varepsilon}\left(\hat{\mathbb{P}}(\boldsymbol{z})\right)}\text{CVaR}_{1-\beta}^{\boldsymbol{y}\sim\mathbb{Q}}\left(g(\boldsymbol{y}, \boldsymbol{z})\right) \leq 0\ ,
		\end{aligned}}
	\end{equation}
	where the ambiguity set 
	\begin{equation*}\scalemath{0.84}{
		\mathcal{B}^{\varepsilon}\left(\hat{\mathbb{P}}(\boldsymbol{z})\right) := \left\{\mathbb{Q}\in \mathcal{M}(\mathcal{Y})\middle| d_{W}\left(\hat{\mathbb{P}}(\boldsymbol{z}),\mathbb{Q}\right)\leq \varepsilon \right\} }
	\end{equation*}
	is defined as a ball of radius $\varepsilon$ with respect to the Wasserstein metric~\eqref{eq:Wasserstein_distance}, centered on the empirical distribution $\hat{\mathbb{P}}(\boldsymbol{z})$.
	Differently from other approaches, e.g. \cite{mark2020stochastic,zhong2021data}, we define the center of the ambiguity set as the empirical distribution of the $T$-step predictions $\hat{\boldsymbol{y}}^i(\boldsymbol{z})$,~\mbox{$i=1,\dots,N$}. As a consequence, the ambiguity set not only depends on the offline collected trajectories and on the approximate predictor, but also on the optimization variable $\boldsymbol{z}$, i.e., on the initial condition $\boldsymbol{x}_0$ and the control inputs sequence $\boldsymbol{u}$. Notice that separately formulating the worst-case in the cost and in the constraints might introduce some conservatism as the solution robustifies against two, possibly different, worst-case distributions.
	
	\section{Main results/Finite sample guarantees}\label{Sec:MainResults}
	We now want to derive a lower bound on the radius $\varepsilon$ that guarantees, with high probability, that the true distribution $\mathbb{P}_{\boldsymbol{y}|\boldsymbol{z}}$ is contained in the ambiguity set $\mathcal{B}^{\varepsilon}(\hat{\mathbb{P}}(\boldsymbol{z}))$. This will allow us to show that if the DR Problem~\eqref{eq:DR_MPC_1} is feasible and the minimizer $\boldsymbol{z}^*$ attains an optimal cost $\hat{J}^{DR}(\boldsymbol{z}^*)$, then, with high confidence, $\boldsymbol{z}^*$ is a feasible solution to the original Problem~\eqref{eq:MPC_problem_2} and the resulting cost $\bar{J}(\boldsymbol{z}^*)$ is upper bounded by $\hat{J}^{DR}(\boldsymbol{z}^*)$.
	
	We require a technical assumption on the distribution of the multi-step noise $\boldsymbol{\xi} = \bar{H} \boldsymbol{w}$.
	
	\begin{Assumption}[Light-tail assumption]\label{Ass:Light_Tail}
		For some constants $a>1$ and $\gamma>0$ 
		\begin{equation*}\scalemath{0.84}{
			\mathscr{E}_{a,b} = \mathbb{E}\left[e^{b \|{\boldsymbol{\xi}}\|^a}\right]< +\infty\ .}
		\end{equation*}
	\end{Assumption}
	\vspace{-2pt}
	This assumption is a condition on the decay rate of the tail of the probability distribution $\mathbb{P}_{w}$ and it is trivially satisfied when $w\sim\mathbb{P}_{w}$ is sub-Gaussian or when $\mathcal{W}$ is compact. 
	
	To help us in the derivation of the upcoming proofs, we define the distribution
	\begin{equation*}\scalemath{0.84}{
		\bar{\mathbb{P}}(\boldsymbol{z}):=\frac{1}{N} \sum_{i=1}^{N} \delta_{\bar{\boldsymbol{y}}^i(\boldsymbol{z})}\ ,}
	\end{equation*}
	where
	\begin{equation*}\scalemath{0.84}{
		\begin{aligned}
			\bar{\boldsymbol{y}}^i(\boldsymbol{z}) :=& \bar{L} \boldsymbol{z} + \bar{\boldsymbol{\xi}}^i,\ \text{ for } i=1,\dots,N\ \ , \\
			\bar{\boldsymbol{\xi}}^i :=& \boldsymbol{y}^i - \bar{L} \boldsymbol{z}^i\ ,\ \text{ for } i=1,\dots,N\ .
		\end{aligned}}
	\end{equation*}
	For given initial condition and control input sequence $\boldsymbol{z}$, the future state prediction $\bar{\boldsymbol{y}}^i(\boldsymbol{z})$, $i=1,\dots,N$, is  computed using the true (unknown) multi-step dynamics $\bar{L}$ and the true multi-step residuals $\bar{\boldsymbol{\xi}}^i$, $i=1,\dots,N$. Since we employed here the true dynamics to derive the residuals from the available data and to compute the prediction, $\bar{\mathbb{P}}(\boldsymbol{z})$ is an empirical sampled version of the true distribution ${\mathbb{P}}_{\boldsymbol{y}|\boldsymbol{z}}$. Note that we introduced these objects for analysis purposes only and, since $\bar{L}$ is unknown, we cannot compute them.
	
	To obtain the ambiguity set radius $\varepsilon$, we start by upper bounding the Wasserstein distance $d_{W}(\hat{\mathbb{P}}(\boldsymbol{z}), \mathbb{P}_{\boldsymbol{y}|\boldsymbol{z}})$ between the empirical distribution obtained with the approximate dynamics $\hat{\mathbb{P}}(\boldsymbol{z})$ and the true unknown distribution $\mathbb{P}_{\boldsymbol{y}|\boldsymbol{z}}$, inspired by~\cite{kannan2020residuals}. We apply the triangle inequality to obtain
	\begin{equation*}\label{eq:1}\scalemath{0.84}{
		d_{W}\left(\hat{\mathbb{P}}(\boldsymbol{z}), \mathbb{P}_{\boldsymbol{y}|\boldsymbol{z}}\right) \leq d_{W}\left(\hat{\mathbb{P}}(\boldsymbol{z}), \bar{\mathbb{P}}(\boldsymbol{z})\right) +d_{W}\left(\bar{\mathbb{P}}(\boldsymbol{z}), \mathbb{P}_{\boldsymbol{y}|\boldsymbol{z}}\right).}
	\end{equation*}
	With the following lemmas we upper bound, uniformly in $\boldsymbol{z}$, the terms appearing on the right-hand side.
		\begin{Lemma}\label{Lemma:2}
			Under Assumption~\ref{Ass:Spectral_norm_bound}, let $\alpha$ specify a risk level with $\alpha \in(0,1)$. Then, for each $\boldsymbol{z} \in \mathcal{Z}_{\tau}$
			\begin{equation*}\label{eq:LemmaII_statement}\scalemath{0.84}{
					\mathbb{P}^{\mathcal{ID}}\left\{d_{W}\left(\hat{\mathbb{P}}(\boldsymbol{z}), \bar{\mathbb{P}}(\boldsymbol{z})\right) \geq \gamma\left(\frac{\alpha}{2}\right)\frac{1}{N} \sum_{i=1}^{N} \left\| \boldsymbol{z}-\boldsymbol{z}^{i} \right\|  \right\} \leq \frac{\alpha}{2}\ .}
			\end{equation*}
		\end{Lemma}
		\begin{proof} \begin{equation*}\scalemath{0.84}{
				\begin{aligned}
					d_{W}\left(\hat{\mathbb{P}}(\boldsymbol{z}), \bar{\mathbb{P}}(\boldsymbol{z})\right)
					& \leq\frac{1}{N} \sum_{i=1}^{N}\left\|\hat{\boldsymbol{y}}^{i}(\boldsymbol{z})-\bar{\boldsymbol{y}}^{i}(\boldsymbol{z})\right\| \\
					& =\frac{1}{N} \sum_{i=1}^{N}\left\|\left(\hat{L}\boldsymbol{z}+\hat{\boldsymbol{\xi}}^{i}\right)-\left(\bar{L}\boldsymbol{z}+\bar{\boldsymbol{\xi}}^{i}\right)\right\| \\
					& =\frac{1}{N} \sum_{i=1}^{N}\left\|\left(\hat{L}\boldsymbol{z}-\bar{L}\boldsymbol{z}\right) +\left(\bar{L}\boldsymbol{z}^{i}-\hat{L}\boldsymbol{z}^{i}\right) \right\|\\
					& \leq\frac{1}{N} \sum_{i=1}^{N}\left(\left\|\hat{L}-\bar{L} \right\| \left\| \boldsymbol{z}-\boldsymbol{z}^{i} \right\| \right)\ ,
				\end{aligned}}
			\end{equation*}
			where the first step follows from the definition of Wasserstein distance and the last step from the fact that the matrix norm is consistent with the vector norm that induced it. Therefore,
			\begin{equation*}\label{eq:LemmaId}\scalemath{0.84}{
				\begin{aligned}
					&\mathbb{P}^{\mathcal{ID}}\left\{d_{W}\left(\hat{\mathbb{P}}(\boldsymbol{z}), \bar{\mathbb{P}}(\boldsymbol{z})\right) \geq \gamma\left(\frac{\alpha}{2}\right)\left(\frac{1}{N} \sum_{i=1}^{N} \left\| \boldsymbol{z}-\boldsymbol{z}^{i} \right\| \right)  \right\} \\
					& \leq \mathbb{P}^{\mathcal{ID}} \left\{\!\left\|\hat{L}\!-\!\bar{L} \right\| \frac{1}{N} \!\sum_{i=1}^{N}\left\| \boldsymbol{z}\!-\!\boldsymbol{z}^{i} \right\|\geq \gamma\!\left(\frac{\alpha}{2}\right)\frac{1}{N}\! \sum_{i=1}^{N} \left\| \boldsymbol{z}\!-\!\boldsymbol{z}^{i} \right\| \! \right\} \leq \frac{\alpha}{2}\ .
				\end{aligned}}
			\end{equation*}
		\end{proof}
		\pagebreak
		The derived bound on the distance $d_{W}(\hat{\mathbb{P}}(\boldsymbol{z}), \bar{\mathbb{P}}(\boldsymbol{z}))$ depends on the model mismatch $\| \hat{L}-\bar{L} \|$ and on the average distance between the $\boldsymbol{z}$ and the collected $\boldsymbol{z}^i$, $i = 1,\dots,N$.
		\begin{Lemma}{(\cite[Theorem 2]{fournier2015rate})}\label{Lemma:fournier}: Under Assumption~\ref{Ass:Light_Tail}, for $nT > 2$, for all $\kappa>0$, $ N \in \mathbb{N}$, and $\boldsymbol{z} \in \mathcal{Z}_{\tau}$
		\begin{equation*}\scalemath{0.84}{
					\mathbb{P}^N \left\{d_{W}\left(\bar{\mathbb{P}}(\boldsymbol{z}), \mathbb{P}_{\boldsymbol{y}|\boldsymbol{z}}\right)\geq \kappa\right\} \leq
					\left\{\begin{aligned}
						&c_1 \exp \left(-c_2 N \kappa^{nT}\right) &\text{if } \kappa \leq 1 \\
						&c_1 \exp \left(-c_2 N \kappa^{a}\right) &\text{if } \kappa>1\
					\end{aligned} \right. , }
			\end{equation*}
			where $\mathbb{P}^N$ is the $N$-fold distribution of the data generating process~\eqref{eq:dynamics_matrix_form}. The positive constants $c_1$ and $c_2$ depend on the dimensions of ${\boldsymbol{\xi}}$ and on the constants $a$, $b$ and $\mathscr{E}_{a,b}$.
		\end{Lemma}
		\vspace{-3pt}
		We are now in the position to prove a finite sample guarantee result.
		\vspace{-3pt}
		\begin{Theorem}\label{Thm:Finite_sample_certificate}
			Under Assumptions~\ref{Ass:Spectral_norm_bound} and~\ref{Ass:Light_Tail}, and for a risk level $\alpha \in(0,1)$, define the radius of the ambiguity set
			\begin{equation}\label{eq:zeta}\scalemath{0.84}{
				\varepsilon(\alpha,\boldsymbol{z}):=\varepsilon_1(\alpha) \left(\frac{1}{N} \sum_{i=1}^{N} \left\| \boldsymbol{z}-\boldsymbol{z}^{i} \right\| \right) +\varepsilon_2(\alpha)\ ,}
			\end{equation}
		with \vspace{-2pt}
			\begin{equation*}\scalemath{0.84}{
				\begin{aligned}
					\varepsilon_1\!(\alpha)\!:=\!\gamma\!\left(\!\frac{\alpha}{2}\!\right)\!, \quad
					\varepsilon_2\!\left(\alpha\right)\!:=\!\left\{ \! \begin{aligned}
						\!&\left(\frac{\log \left(c_1 \frac{2}{\alpha}\right)}{c_2 N}\right)^{\frac{1}{nT}} \text {if } \!N\! \geq \! \frac{\log\! \left(c_1 \frac{2}{\alpha}\right)}{c_2 N} \\
						\!&\left(\frac{\log \left(c_1 \frac{2}{\alpha}\right)}{c_2 N}\right)^{\frac{1}{a}} \text {if } \!N\! <\! \frac{\log\! \left(c_1 \frac{2}{\alpha}\right)}{c_2 N}
					\end{aligned}\right. \!.
				\end{aligned}}
			\end{equation*}
			Let $\hat{J}^{DR}\left(\boldsymbol{z}^{*}\right)$ and $\boldsymbol{z}^{*}$ be the optimal value and a feasible optimizer of the distributionally robust FHOC Problem~\eqref{eq:DR_MPC_1} with decision-dependent ambiguity set $\mathcal{B}^{\varepsilon}(\hat{\mathbb{P}}(\boldsymbol{z}))$ of radius $\varepsilon=\varepsilon(\alpha,\boldsymbol{z})$ as defined in~\eqref{eq:zeta}. Then,
			\begin{equation*}\scalemath{0.84}{
				\left\{\begin{aligned}
				\ &\mathbb{P}\left\{\bar{J}\left(\boldsymbol{z}^{*}\right) \leq \hat{J}^{DR}\left(\boldsymbol{z}^{*}\right)\right\} \geq 1-\alpha \\
				&\mathbb{P}\left\{\text{CVaR}_{1-\beta}^{\boldsymbol{y}\sim\mathbb{P}_{\boldsymbol{y}|\boldsymbol{z}^*}}\left(g(\boldsymbol{y}, \boldsymbol{z}^*)\right) \leq 0 \right\} \geq 1-\alpha
				\end{aligned}\right.\ ,}
			\end{equation*}
		where $\mathbb{P}=\mathbb{P}^{\mathcal{ID}}\times\mathbb{P}^N$.
		\end{Theorem}
		\begin{proof}
			The claim is a direct consequence of Lemma~\ref{Lemma:fournier} and Lemma~\ref{Lemma:2}. Choosing $\varepsilon=\varepsilon(\alpha,\boldsymbol{z})$ as in~\eqref{eq:zeta} ensures that 
			\begin{equation*}\scalemath{0.84}{
				\begin{aligned}
					&\mathbb{P}\left\{d_{W}\left(\hat{\mathbb{P}}(\boldsymbol{z}), \mathbb{P}_{\boldsymbol{y}|\boldsymbol{z}}\right)\geq \varepsilon(\alpha,\boldsymbol{z}) \right\} \leq\\
					&\quad \leq \mathbb{P}^{\mathcal{ID}}\Big\{d_{W}\left(\hat{\mathbb{P}}(\boldsymbol{z}), \bar{\mathbb{P}}(\boldsymbol{z})\right) \geq\varepsilon_1(\alpha)\frac{1}{N} \sum_{i=1}^{N} \left\| \boldsymbol{z}-\boldsymbol{z}^{i} \right\|  \Big\} \ + \\
					&\quad\quad\quad\quad +\mathbb{P}^{N}\left\{d_{W}\left(\bar{\mathbb{P}}(\boldsymbol{z}), \mathbb{P}_{\boldsymbol{y}|\boldsymbol{z}}\right)\geq \varepsilon_2(\alpha) \right\}\\
					&\quad \leq \frac{\alpha}{2} + \frac{\alpha}{2} =  \alpha\ .
				\end{aligned}}
			\end{equation*}
		The claim follows from the definition of the ambiguity set $\mathcal{B}^{\varepsilon}(\hat{\mathbb{P}}(\boldsymbol{z}))$ and of the DR Problem~\eqref{eq:DR_MPC_1}.
		\end{proof}
		
		The radius in~\eqref{eq:zeta} is the sum of two components, one related to limited number of available samples through Lemma~\ref{Lemma:fournier} and the other to the error in the system dynamics through Lemma~\ref{Lemma:2}.
		While the derivation of $\varepsilon(\alpha,\boldsymbol{z})$ helps us derive theoretical guarantees, it typically leads to conservative bounds in practice. We investigate a practical data-driven approach for choosing the radius $\varepsilon(\alpha,\boldsymbol{z})$ in Section~\ref{Sec:Numerics}.
		
		The supremums over probability distributions appearing in the DR Problem~\eqref{eq:DR_MPC_1} make it an infinite-dimensional optimization program. In the following theorem, leveraging recent results in the DR optimization literature, we show that, for the class of problems with convex piecewise affine cost and constraint functions, Problem~\eqref{eq:DR_MPC_1} can be reformulated as a tractable, finite-dimensional convex program.
		\pagebreak
		\begin{Theorem}\label{Thm:Tractabe_worst_case}
			Let Assumption~\ref{Ass:Spectral_norm_bound} and Assumption~\ref{Ass:Light_Tail} hold. Let $\boldsymbol{y} \!\in\! \mathbb{R}^{nT}$, consider a cost function $h\left(\boldsymbol{y}, \boldsymbol{z} \right) \!=\! \max_{j\leq N_j} h_j(\boldsymbol{y},\boldsymbol{z})$, with $h_j(\boldsymbol{y},\boldsymbol{z})\!=\! a_{j} \boldsymbol{y} + b_{j} \boldsymbol{z} + c_j$ and a constraint function $g\left(\boldsymbol{y}, \boldsymbol{z} \right) \!=\! \max_{k\leq N_k} g_k(\boldsymbol{y},\boldsymbol{z})$, with $g_k(\boldsymbol{y},\boldsymbol{z}) \!=\! d_{k} \boldsymbol{y} + e_{k} \boldsymbol{z} + f_k$.
			Then, the DR FHOC Problem~\eqref{eq:DR_MPC_1} can be formulated as
			\begin{equation*}\scalemath{0.84}{
				\begin{aligned}
					\hat{J}^{DR}=
						\inf_{\boldsymbol{z}, s_{i}}\ &\underline{\lambda} \left(\varepsilon_1(\alpha)\frac{1}{N} \sum_{i=1}^{N} \left\| \boldsymbol{z}-\boldsymbol{z}^{i} \right\| + \varepsilon_2(\alpha)\right)+\frac{1}{N} \sum_{i=1}^{N} s_{i} \\
						\text{ s.t. } & \boldsymbol{z} \in \mathcal{Z}_{CVaR} \\
						& a_{j}\! \left(\!\hat{L}\boldsymbol{z} + \hat{\boldsymbol{\xi}}^{i}\!\right) + b_{j} \boldsymbol{z} +c_j \leq s_{i} \\ &{\forall\  i\!=\!1,\dots,N,\  j\!=\!1,\dots,N_j}\ ,
					\end{aligned}}
			\end{equation*}
			with 
			\begin{equation*}\scalemath{0.84}{
				\begin{aligned}
					&\ {\mathcal{Z}}_{\mathrm{CVaR}}=\\
					& \left\{ \boldsymbol{z}\! \in\! \mathcal{Z}_{\tau}\, \middle|\, \begin{aligned}
						& \exists\ t, q_i \text{  such that }\\
						& \underline{\theta} \left(\varepsilon_1(\alpha)\frac{1}{N} \sum_{i=1}^{N} \left\| \boldsymbol{z}-\boldsymbol{z}^{i} \right\| + \varepsilon_2(\alpha)\right) + \frac{1}{N}\! \sum_{i=1}^{N} \!q_{i} \leq t \beta \\
						&\left(d_{k} \left(\hat{L}\boldsymbol{z} + \hat{\boldsymbol{\xi}}^{i}\right)+ e_{k} \boldsymbol{z} + f_k +t \right)_{+} \leq q_{i} \\
						&\forall\  i=1,\dots,N,\  k=1,\dots,N_k\end{aligned}
					\right\} \ ,\\
					&\quad\underline{\lambda} = \max_{j\leq N_j} \left\|a_{j}\right\|,\quad \underline{\theta} = \max_{k\leq N_k} \left\|d_{k}\right\|.\\
				\end{aligned}}
			\end{equation*}
		\end{Theorem}
		\begin{proof}
			For a convex piecewise affine cost function $h\left(\boldsymbol{y}, \boldsymbol{z} \right)$, we can apply Corollary 5.1 of~\cite{mohajerin2018data} to reformulate the supremum appearing in the worst-case expectation of~\eqref{eq:DR_MPC_1} as
			\begin{equation}\label{eq:Thm7_proof_worst_case_expectation}\scalemath{0.84}{
				 	\begin{aligned}
					\sup_{\mathbb{Q} \in \mathcal{B}^{\varepsilon}\!\left(\hat{\mathbb{P}}(\boldsymbol{z})\right)}\! \mathbb{E}^{\boldsymbol{y}\sim\mathbb{Q}}\left[ h\left(\boldsymbol{y}, \boldsymbol{z} \right) \right] =
						\inf _{\lambda, s_{i}}\ &\lambda\, \varepsilon(\alpha,\boldsymbol{z})+\frac{1}{N} \sum_{i=1}^{N} s_{i} \\
						\text{ s.t. } & a_{j} \hat{\boldsymbol{y}}^{i}(\boldsymbol{z}) + b_{j} \boldsymbol{z} +c_j \leq s_{i} \\ 
						& \left\|a_{j}\right\| \leq \lambda\\
						& {\forall\  i\!=\!1,\dots,N,\  j\!=\!1,\dots,N_j}\ .
					\end{aligned}}
			\end{equation}
			From the definition of CVaR, the constraints in Problem~\eqref{eq:DR_MPC_1} can be written as the set
			\begin{equation*}\scalemath{0.84}{
				\begin{aligned}
					&{\mathcal{Z}}_{\mathrm{CVaR}}=\!\left\{ \boldsymbol{z} \in \mathcal{Z}_{\tau} \middle| \sup_{\mathbb{Q} \in \mathcal{B}^{\varepsilon}\left(\hat{\mathbb{P}}(\boldsymbol{z})\right)}\text{CVaR}_{1-\beta}^{\boldsymbol{y}\sim\mathbb{Q}}\left(g(\boldsymbol{y}, \boldsymbol{z})\right) \leq 0 \right\}\\
					& = \left\{ \boldsymbol{z} \in \mathcal{Z}_{\tau} \middle| \sup_{\mathbb{Q} \in \mathcal{B}^{\varepsilon}\left(\hat{\mathbb{P}}(\boldsymbol{z})\right)} \inf_{t \in \mathbb{R}} \left[\mathbb{E}^{\boldsymbol{y}\sim\mathbb{Q}}\left[(g\left(\boldsymbol{y}, \boldsymbol{z} \right)+t)_{+}\right]-t \beta\right] \leq 0 \right\} .
				\end{aligned}}
			\end{equation*}
			
			For a convex piecewise affine constraint function $g\left(\boldsymbol{y}, \boldsymbol{z} \right)$, following Proposition V.1 of~\cite{hota2019data}, we can reformulate the supremum appearing in the constraint as
			\begin{equation}\label{eq:Thm7_proof_constrints}\scalemath{0.84}{
					{\mathcal{Z}}_{\mathrm{CVaR}}=\left\{ \boldsymbol{z} \in \mathcal{Z}_{\tau} \middle| \begin{aligned}
						& \exists\ t, \theta, q_i \text{  such that } \\
						& \theta \varepsilon(\alpha,\boldsymbol{z}) + \frac{1}{N}\! \sum_{i=1}^{N} \!q_{i} \leq t \beta \\
						&\left(d_{k} \hat{\boldsymbol{y}}^{i}(\boldsymbol{z}) + e_{k} \boldsymbol{z} + f_k +t \right)_{+} \leq q_{i} \\
						&\left\|d_{k}\right\| \leq \theta \\
						&\forall\  i=1,\dots,N,\  k=1,\dots,K_k\end{aligned}
					\right\}\ .}
				\end{equation}
							
				Substituting~\eqref{eq:Multistep_Prediction} in~\eqref{eq:Thm7_proof_worst_case_expectation} and~\eqref{eq:Thm7_proof_constrints}, and taking the infimum over $\boldsymbol{z}$ results in the finite-dimensional convex program
				\begin{equation*}\scalemath{0.84}{
					\begin{aligned}\hat{J}^{DR} =
						\inf _{\boldsymbol{z},s_{i}}\ & \underline{\lambda}\, \varepsilon(\alpha,\boldsymbol{z})+\frac{1}{N} \sum_{i=1}^{N} s_{i} \\
						\text { s.t. } & \boldsymbol{z} \in {\mathcal{Z}}_{\mathrm{CVaR}} \\
						& a_{j}\! \left(\!\hat{L}\boldsymbol{z} + \hat{\boldsymbol{\xi}}^{i}\!\right) + b_{j} \boldsymbol{z} +c_j \leq s_{i} \\
						& \underline{\lambda} \!=\! \max_{j\leq N_j} \left\|a_{j}\right\| \quad{\forall\  i\!=\!1,\dots,N,\  j\!=\!1,\dots,N_j} 
					\end{aligned}}
				\end{equation*}
				with
				\begin{equation*}\scalemath{0.84}{
					{\mathcal{Z}}_{\mathrm{CVaR}}\!=\!\left\{ \boldsymbol{z}\! \in\! {\mathcal{Z}}_{\tau} \middle|\, \begin{aligned}
						& \exists\ t, q_i \text{  such that }\\
						& \underline{\theta} \varepsilon(\alpha,\boldsymbol{z}) + \frac{1}{N}\! \sum_{i=1}^{N} \!q_{i} \leq t \beta \\
						&\left(d_{k} \left(\hat{L}\boldsymbol{z} + \hat{\boldsymbol{\xi}}^{i}\right)+ e_{k} \boldsymbol{z} + f_k +t \right)_{+} \leq q_{i} \\
						&\underline{\theta} = \max_{k\leq N_k} \left\|d_{k}\right\| \\
						&\forall\  i=1,\dots,N,\  k=1,\dots,N_k\end{aligned}
					\right\}.}
				\end{equation*}
				Since $\boldsymbol{y} \!\in\! \mathbb{R}^{nT}$, $\underline{\lambda}$ and $\underline{\theta}$ are not optimization variable, leading to an exact convex reformulation.
				As the bound on the radius $\varepsilon(\alpha,\boldsymbol{z})$ holds uniformly in $\boldsymbol{z}$, the claim follows.
			\end{proof}
			
			\section{Numerical example}\label{Sec:Numerics}
			We consider the system
			\begin{equation*}\label{eq:Dynamics_example}\scalemath{0.84}{
				\boldsymbol{x}_{k+1}=
				\begin{bmatrix} 
					0.9 & 0.1\\
					0.05 & 0.9\end{bmatrix} \boldsymbol{x}_{k} + 
				\begin{bmatrix}
					0 \\
					1 \\
				\end{bmatrix} \boldsymbol{u}_{k} +\boldsymbol{w}_k,}
			\end{equation*}
			with additive disturbance $\boldsymbol{w}_k \sim \mathcal{N}(0,0.03^2)$.
			We assume that we have access to a dataset $\mathcal{D}^{N,T}$ comprising $N$ trajectories of length $T=5$, collected by applying a random input sequence \mbox{$\boldsymbol{u}^i_k \sim \mathcal{N}(0,0.5^2)$}, $k=1,\dots,T$, $i=1,\dots,N$, starting from random initial conditions $\boldsymbol{x}^i_0\sim\mathcal{N}(0,0.5^2)$, $i=1,\dots,N$. We consider cost and constraint functions
			\begin{equation}\label{eq:cost_fun_example}\scalemath{0.84}{
				h\left(\boldsymbol{y}, \boldsymbol{z} \right) = \|\boldsymbol{y}_{(1)}-\mathds{1}\|_1 \ ,}
			\end{equation}
			\begin{equation}\label{eq:constr_fun_example}	\scalemath{0.84}{		
					g\left(\boldsymbol{y}, \boldsymbol{z} \right) = \max\{ [ (\boldsymbol{y}_{(1)}-\mathds{1})^\top, \boldsymbol{y}_{(2)}^\top ]^\top\}\ , }
			\end{equation}		
		where $\boldsymbol{y}_{(1)}$ and $\boldsymbol{y}_{(2)}$ represent the first and the second entry of the predicted future states, $\| \!\cdot\! \|_1 $ is the $1$-norm, and $\mathds{1}$ is a vector of ones of appropriate dimensions. The cost function~\eqref{eq:cost_fun_example} requires the first entry of the state to stay as close as possible to a constant reference of $1$, while the constraint~\eqref{eq:constr_fun_example} enforces $\boldsymbol{y}_{(1)}$ to be smaller than $1$ and $\boldsymbol{y}_{(2)}$ to be positive through a CVaR constraint of level $\beta\!=\!0.2$. To guarantee the existence of a feasible solution at each MPC step, we consider a slack formulation for the CVaR constraint, with the slack variable linearly weighted in the cost function with a weight of $10^6$. For all the simulations, the initial condition is set to $\boldsymbol{x}_0 \!=\! [0.9,0.9]^\top $ and the MPC horizon is set to $30$ steps.
			
		To study the effect of different ambiguity set radius on the closed-loop cost and violations, we first fix a dataset of size of $N=10$ and a closed-loop noise realization. Instead of assuming a given $\hat{L}$, we use the data in $\mathcal{D}^{N,T}$ to compute a $T$-step linear predictor by solving the following least-squares minimization problem
			\begin{equation}\label{eq:Multistep_LeasSquares_Regression}\scalemath{0.84}{
				\begin{aligned}
					\hat{L}:=\arg\min_{L} \quad & \sum_{i=1}^N \left\| L \boldsymbol{z}^i - \boldsymbol{y}^i \right\|^2\ ,
				\end{aligned}}
			\end{equation}
			under the assumption that the collected data is sufficiently informative. To improve the quality of the predictor, we reduce the number of unknowns by enforcing on $\hat{L}$ the block triangular causal structure of $\bar{L}$. With the identified dynamics $\hat{L}$ we compute the $N$ $T$-step residuals as in~\eqref{eq:Multistep_Residuals}. We intentionally choose a low signal-to-noise ratio and a small number of trajectories in the offline-collected dataset to highlight how limited data availability and model mismatch (in this case due to the identification process) can hamper MPC performance. We then proceed to solve the DRMPC for all the possible combinations of ${\varepsilon_1,\varepsilon_2}\in\{10^{\minus7},10^{\minus6},10^{\minus5},10^{\minus4},10^{\minus3},10^{\minus2},10^{\minus1},10^{0}\}$. The resulting closed-loop cost and violations for different values of the ambiguity set radius are shown in Figure~\ref{fig:ClosedLoopCostAndViolHeatmap}. For small values of the radius, the controller experiences a large number of constraint violations, as the radius of the ambiguity set is increased, the number of violations is reduced at the expense of higher closed-loop costs. 
			\begin{figure}
				\includegraphics[width=0.99\columnwidth]{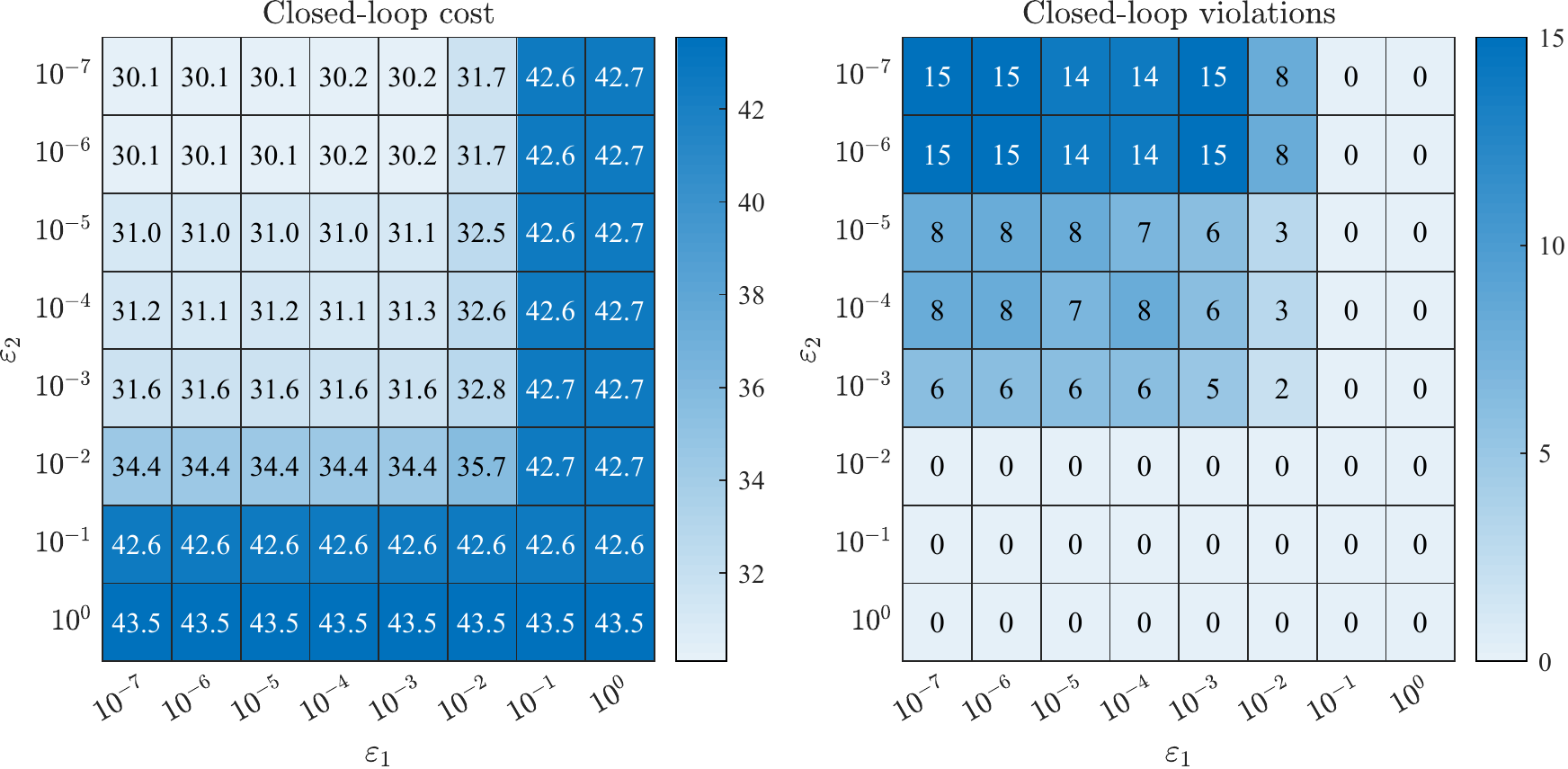}
				\caption{Closed-loop cost (left) and violations (right) for different values of $\varepsilon_1$ and $\varepsilon_2$.}\label{fig:ClosedLoopCostAndViolHeatmap}
			\end{figure}
			
			As mentioned in \cite{dean2020sample}, the theoretical bound on $\gamma(\alpha)$ in Assumption~\ref{Ass:Spectral_norm_bound} can be quite conservative in practice. Moreover, the use of the triangle inequality in Theorem~\ref{Thm:Finite_sample_certificate} can introduce some extra conservatism, leading to a loose bound on the radius $\varepsilon(\alpha,\boldsymbol{z})$. Informed by the expression in~\eqref{eq:zeta}, we aim to obtain a less conservative ambiguity set radius by directly exploiting data to approximate the Wasserstein distance $d_{W}(\hat{\mathbb{P}}(\boldsymbol{z}), \mathbb{P}_{\boldsymbol{y}|\boldsymbol{z}})$.
			
				\begin{algorithm}[ht]\label{Algo:LeaveOneOut_1}
				\caption{Data-driven empirical radius estimate}
				\smaller \begin{algorithmic}[1]
					\renewcommand{\algorithmicrequire}{\textbf{Input:}}
					\renewcommand{\algorithmicensure}{\textbf{Output:}}
					\REQUIRE $\mathcal{D}^{N,T}$.
					\ENSURE $\hat{L}$, $\varepsilon_1$, $\varepsilon_2$. 
					\FOR {$\ell=1$ to $N$}
					\STATE define $\mathcal{D}^{N\!\minus\! 1,T}_{I_{\ell}}$ the leave-one-out dataset with indices $I_{\ell}=\{1,\dots,N\}\backslash{\ell}$;
					\STATE compute $\tilde{L}^{\ell}$ from $\mathcal{D}^{N\!\minus\!1,T}_{I_{\ell}}$ by least-squares regression as in \eqref{eq:Multistep_LeasSquares_Regression};
					\STATE compute residuals $\tilde{\boldsymbol{\xi}}^i_{\ell} := \boldsymbol{y}^i - \tilde{L}_{\ell} \boldsymbol{z}^i$, $i=1,\dots,N$;
					\STATE compute $ V_{\ell}=\frac{1}{N-1} \sum_{i\in I_{\ell}}\left\| \boldsymbol{z}^{\ell}- \boldsymbol{z}^{i} \right\|$;
					\STATE compute $ E_{\ell}=\frac{1}{N^2-N} \sum_{i\in I_{\ell}} \left\| {\boldsymbol{y}}^{\ell}- \tilde{\boldsymbol{y}}^i(\boldsymbol{z}^{\ell}) \right\|$ with $\tilde{\boldsymbol{y}}^i(\boldsymbol{z}^{\ell}) = \tilde{L}_{\ell} \boldsymbol{z}^{\ell} + \tilde{\boldsymbol{\xi}}^i_{\ell}$;
					\ENDFOR
					\RETURN $\hat{L}=\frac{1}{N} \sum_{\ell =1}^{N}\tilde{L}_{\ell}$ and $\{\varepsilon_1\geq 0, \varepsilon_2\geq0\}$ that minimize $\sum_{{\ell}=1}^{N} \left\| \varepsilon_1 V_{\ell} + \varepsilon_2 - E_{\ell} \right\|$.
				\end{algorithmic}\normalsize
			\end{algorithm}
		
			In Algorithm~\ref{Algo:LeaveOneOut_1} we use a leave-one-out procedure to obtain an estimate of the dynamics and of the ambiguity set parameters $\varepsilon_1$ and $\varepsilon_2$. For each leave-one-out dataset, $E_{\ell}$ is the Wasserstein distance $d_{W}( \tilde{\mathbb{P}}_{i\in I_{\ell}}(\boldsymbol{z}^{\ell}), \tilde{\mathbb{P}}(\boldsymbol{z}^{\ell}))$ between the leave-one-out empirical distribution $\tilde{\mathbb{P}}_{i\in I_{\ell}}(\boldsymbol{z}^{\ell})\!:=\!\frac{1}{N-1} \sum_{i\in I_{\ell}} \delta_{\tilde{\boldsymbol{y}}^i(\boldsymbol{z}^{\ell})}$ and the one that uses all the samples $\tilde{\mathbb{P}}(\boldsymbol{z})\!:=\!\frac{1}{N} \sum_{i=1}^{N} \delta_{\tilde{\boldsymbol{y}}^i(\boldsymbol{z}^{\ell})}$ in $\boldsymbol{z}\!=\!\boldsymbol{z}^j$. The only difference between the two distributions is that $\tilde{\mathbb{P}}(\boldsymbol{z}^{\ell})$ contains one extra impulse in $\tilde{\boldsymbol{y}}^{\ell}(\boldsymbol{z}^{\ell})$. As the predictions are computed for $\boldsymbol{z}\!=\!\boldsymbol{z}^{\ell}$, by construction we have $\tilde{\boldsymbol{y}}^{\ell}(\boldsymbol{z}^{\ell})\!=\!{\boldsymbol{y}}^{\ell}$ independently of the dynamics, making $E_{\ell}$ the distributional mismatch that we can expect from a new independent observation~$(\boldsymbol{z}^{\ell},\boldsymbol{y}^{\ell})\sim~\mathbb{P}_{\boldsymbol{y}|\boldsymbol{z}=\boldsymbol{z}^{\ell}}$.
			
			\begin{figure}[b]
				\includegraphics[width=1\columnwidth]{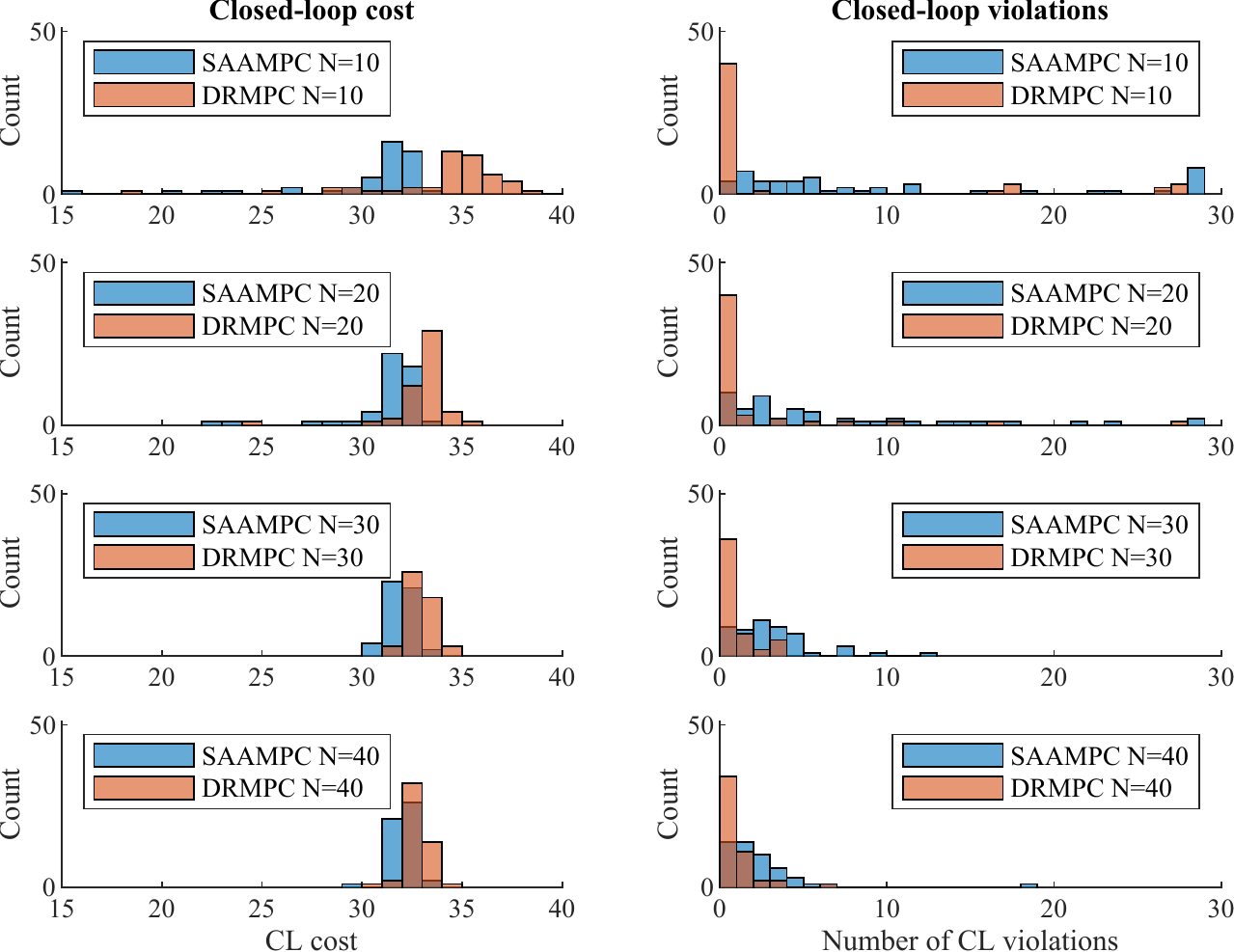}
				\caption{Comparison of Closed-loop cost (left) and Closed-loop violations (right) for dataset sizes $N=\{10,20,30,40\}$ between SAAMPC and DRMPC with radius tuned as in Algorithm~\ref{Algo:LeaveOneOut_1}.}\label{fig:Closed_Loop_SAAtoDR_cost_viol}
			\end{figure}
			
			In Figure~\ref{fig:Closed_Loop_SAAtoDR_cost_viol} we compare the closed-loop cost and Closed-loop violations for the proposed DR approach with data-driven radius obtained as in Algorithm~\ref{Algo:LeaveOneOut_1} (DRMPC) and for SAA approach~\eqref{eq:MPC_problem_SAA} (SAAMPC) that is obtained by setting $\varepsilon(\alpha, \boldsymbol{z})=0$. We perform simulations for $N\!=\!\{10,20,30,40\}$, each repeated for $50$ realizations of the identification dataset and closed-loop noise trajectory. We can observe that the DRMPC approach can reduce the chance of closed-loop constraint violations, with only a limited increase in closed-loop cost. As $N$ increases, the uncertainty decreases, and we observe a reduction in the performance difference between DRMPC and SAAMPC for both closed-loop cost and constraint violations.
			
			\section{Conclusions}\label{Sec:Conclusion}
			We presented a novel data-driven DRMPC formulation for unknown discrete-time linear time-invariant systems affected by additive uncertainties. We obtained finite sample probabilistic guarantees on the worst-case expectation and CVaR constraint in the presence of uncertainty on both the model of the dynamics and on the disturbance distribution. We then derived a finite-dimensional tractable reformulation of the DR problem for convex piecewise affine cost and constraint functions. Finally, we described a simple data-driven algorithm to obtain an empirical ambiguity set radius estimate and tested the proposed DRMPC against the SAA formulation. The numerical simulations demonstrated the effectiveness of the proposed DRMPC, that achieved a strong reduction in the number of closed-loop constraint violations without substantial increase in the attained closed-loop cost, even when very limited information regarding the dynamics and the disturbance is available.
			Future work focuses on the analysis of the DRMPC closed-loop properties such as stability and recursive feasibility.

			\bibliographystyle{IEEEtran}
			\bibliography{BIB}

		\end{document}